\documentclass{article}

\title{Unavoidable patterns in locally balanced colourings}
\author{Nina Kam\v{c}ev\thanks{Department of Mathematics, Faculty of Science, University of Zagreb, 
\texttt{nina.kamcev@math.hr}. Research supported by the European Union’s Horizon 2020 research and innovation programme [MSCA GA No 101038085].}
\and
Alp M\"uyesser \thanks{
University College London, \texttt{alp.muyesser.21@ucl.ac.uk}}}

\usepackage{graphicx}
\usepackage{amssymb}
\usepackage{amsthm}
\usepackage{amsmath}
\usepackage{enumitem}
\usepackage{framed}
\usepackage{soul}
\usepackage{changepage}
\usepackage{comment}
\usepackage{hyperref}[colorlinks=true]
\usepackage{tikz}
\usetikzlibrary{patterns}
\usepackage{xcolor}
\usepackage[width=1.1\textwidth]{caption}

\theoremstyle{plain}
\newtheorem{theorem}{Theorem}[section]
\newtheorem{lemma}[theorem]{Lemma}
\newtheorem{proposition}[theorem]{Proposition}
\newtheorem{claim}{Claim}
\newtheorem*{claim*}{Claim}

\theoremstyle{definition}

\newtheorem{question}[theorem]{Question}

\newtheorem*{rep@theorem}{\rep@title}
\newcommand{\newreptheorem}[2]{%
\newenvironment{rep#1}[1]{%
 \def\rep@title{#2 \ref{##1}}%
 \begin{rep@theorem}}%
 {\end{rep@theorem}}}

 \newreptheorem{theorem}{Theorem}

\newcommand\eps{\varepsilon}

\newcommand\Fc{\mathcal{F}}

    \newcommand{\Ac}{\mathcal{A}}
    
    \newcommand{\Hc}{\mathcal{H}}
    \newcommand{\Lc}{\mathcal{L}}

\def \ptm {P_{3}^{\circ}}
\def \gammep {\eps}

\newcommand{\pr}[1]{\mathbb{P} \left[ #1 \right]}
\newcommand{\er}[1]{\mathbb{E} \left[ #1 \right]}

\begin{document}
\maketitle
    
    \begin{abstract}  Which patterns must a two-colouring of $K_n$ contain if each vertex has at least $\eps n$ red and $\eps n$ blue neighbours? In this paper, we investigate this question and its multicolour variant.  For instance, we show that any such graph contains a $t$-blow-up of an \textit{alternating 4-cycle} with $t = \Omega(\log n)$. 
    
    \end{abstract}

    \section{Introduction}
    Ramsey's theorem guarantees that any $r$-edge-colouring of a large complete graph yields a large monochromatic complete subgraph. In general, we cannot guarantee the existence of anything but monochromatic subgraphs. Indeed, nothing prevents the host graph from being monochromatic itself. However, in recent years, there have been many results stating that in colourings where each colour is well-represented, a richer family of patterns can be guaranteed.  The following, initially suggested by Bollob\'as \cite{CM}, is a prototypical result of this form. Here, $\mathcal{F}_t$ denotes the family of $2$-edge-coloured complete graphs on $2t$ vertices in which one colour forms either a complete bipartite graph with $t$ vertices on each side $(K_{t,t})$ or a clique of order $t$ $(K_t)$. See Figure~\ref{fig:unavoidable2patterns} for a depiction of the  four colourings in $\mathcal{F}_t$ (up to isomorphism).
    
    \begin{figure}[h]
        \centering
        \includegraphics[scale=0.7]{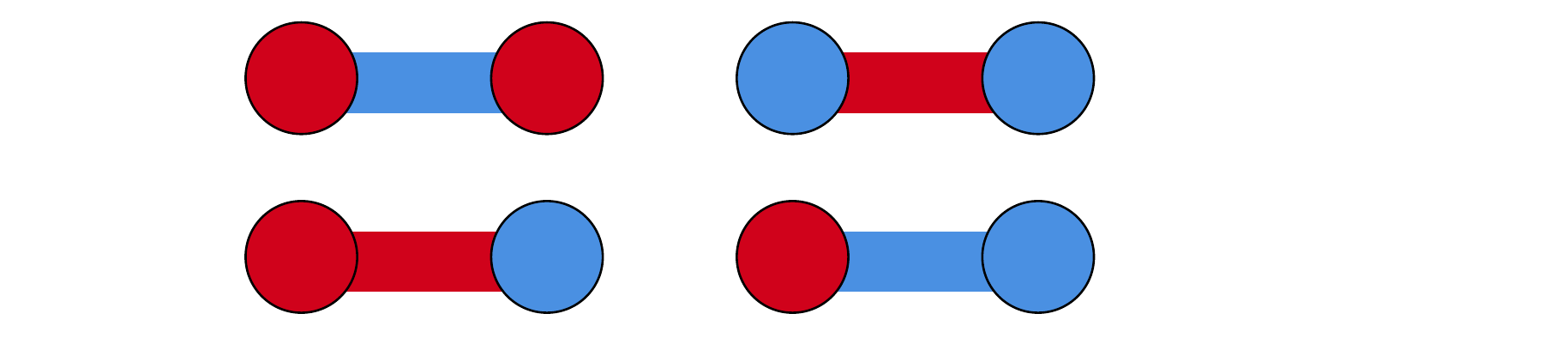}
        \caption{The $4$ types of colourings in $\mathcal{F}_t$. The circles denote cliques of size $t$, whereas the lines connecting the circles denote complete bipartite graphs. In the first row, one colour class forms a $K_{t,t}$, whereas on the second row, one colour class forms a $K_t$.}
        \label{fig:unavoidable2patterns}
    \end{figure}
    
    \begin{theorem}[Fox-Sudakov, Cutler-Mont\'agh \cite{FS, CM}]\label{thm:foxsudakov}
        Let $G$ be a $2$-edge-coloured $K_n$ where each colour class has at least $\eps n^2$ edges, and suppose $n\geq (1/\eps)^{16 t}$. Then, $G$ contains a member of $\mathcal{F}_t$.
    \end{theorem}
     A probabilistic construction shows that the above is asymptotically tight up to the constant factor in the exponent. The theorem is also optimal from a structural standpoint. Namely,  $2$-edge-colourings of $K_n$ where one colour class forms two disjoint cliques of size $\lfloor n/2 \rfloor$ and $\lceil n/2 \rceil$ or one colour class forms a $K_{cn}$ (with $c \sim \sqrt{2}/2$) satisfy the hypothesis of Theorem~\ref{thm:foxsudakov} with $\eps \sim 1/2$, and thus certify that one cannot hope to find patterns which are more complex then those in Figure~\ref{fig:unavoidable2patterns}. Similarly, Theorem~\ref{thm:foxsudakov} becomes false if we delete any of the four patterns contained in $\mathcal{F}_t$. 
    \par There are numerous extensions of Theorem~\ref{thm:foxsudakov}, including multi-colour and infinite variants \cite{BLM}, variants where $K_n$ is replaced by another dense host subgraph \cite{muyesser2020turan}, and variants where $\eps$ is allowed to depend on $n$ \cite{girao2022turan, caro2021unavoidable, girao2022two}. In this paper, we are concerned with the following question, initially raised by Wesley Pegden (personal communication).
    \begin{question}\label{question:mainquestion} Let $\eps>0$. Suppose that $K_n$ is $2$-edge-coloured so that each vertex is incident to at least $\eps n$ edges in each colour. Which subgraphs must $K_n$ necessarily contain?
    \end{question} 
    We call $2$-coloured $K_n$ as in Question~\ref{question:mainquestion} \textbf{locally $\eps$-balanced}. Of course, this is a stronger hypothesis on the colouring compared to Theorem~\ref{thm:foxsudakov}. Hence, any locally $\eps$-balanced $K_n$ must contain a complete subgraph on $2t$ vertices where one colour forms a $K_{t,t}$ or a $K_t$, where $t = \Omega(\log n)$. However, this is not necessarily a complete answer to Question~\ref{question:mainquestion} as it cannot be the case that in a locally $\eps$-balanced $K_n$ one colour class consists entirely of a clique. This motivates the following question: does any locally $\eps$-balanced $K_n$ contain a complete  $2t$-vertex subgraph in which one colour class forms a $K_{t,t}$?

    \par The answer turns out to be negative for any value of $\eps\leq 1/4$ and $t\geq 2$. To see this, consider the following $2$-colouring of $K_{4k}$, where each vertex is adjacent to at least $k$ red and blue edges, illustrated in Figure~\ref{fig:allthegraphs} (labelled $P_3$). Partition $V(K_{4k})$ into four equal-sized sets as $V_1,V_2,V_3$ and $V_4$. Colour all edges with both endpoints in $V_1\cup V_4$ red, colour all edges with both endpoints in $V_2\cup V_3$ blue, and colour edges between $V_1$, $V_3$, and between $V_2$, $V_4$ red, and colour the remaining edges (between $V_1$, $V_2$ and between $V_3$,$V_4$) blue. Denote by $\mathcal{P}_k$ the resulting $2$-coloured $K_{4k}$ (so the number of vertices is $n=4k$). It is easy to see that $\mathcal{P}_k$ contains no complete subgraph with $4$ vertices in which one colour class form two disjoint cliques of size $2$ each. One can also check that $\mathcal{P}_k$ is locally $1/4$-balanced.
    
    \par Perhaps surprisingly, $\mathcal{P}_k$ is an optimal construction in the following sense. For any $\eps>0$, a locally $(1/4+\eps)$-balanced $K_n$ must contain \textit{large} complete subgraphs where one colour class forms a complete bipartite graph. Our first result makes this precise, thereby answering Question~\ref{question:mainquestion} when $\eps>1/4$.
    
    \begin{theorem}\label{thm:quarterbalanced}
        Let $G$ be a $2$-coloured locally $(1/4+\eps)$-balanced $K_n$, and suppose $n\geq 2^{t \cdot 2^{-C\log(1/\eps)^8}}$ for some absolute constant $C$. Then, $G$ contains a complete subgraph on $2t$ vertices where one colour class forms a $K_{t,t}$. 
    \end{theorem}
    \par The earlier construction ($\mathcal{P}_k$) shows that Theorem~\ref{thm:quarterbalanced} does not hold when $1/4+\eps$ is replaced by $1/4$. Moreover, $n$ has to be exponentially large in $t$ for the conclusion to hold (e.g.~by considering a random colouring), but we believe that the dependence of $n$ on $\eps$ should be far from optimal. We discuss this further in Section~\ref{sec:concludingremarks}.
    \input{tikz-patterns}
    \par We now turn our attention to answering Question~\ref{question:mainquestion} when $\eps\leq 1/4$. In light of the earlier construction ($\mathcal{P}_k$), it is not clear at all if anything interesting can be said about this case. One might guess there exists a $2$-colouring of $K_n$, say $\mathcal{P}'_k$, that is locally $1/8$-balanced, but $\mathcal{P}'_k$ does not contain a complete subgraph isomorphic to $\mathcal{P}_2$, or a complete graph on $4$ vertices where one colour class forms two disjoint cliques of size $2$. And perhaps, there exists another colouring which is locally $1/16$-balanced, which does not contain $\mathcal{P}_2'$ or any of the former patterns. Hence, maybe, the answer to Question~\ref{question:mainquestion} is a family of patterns $\mathcal{F}_\eps$ which increases with $\eps^{-1}$. Indeed, Theorem~\ref{thm:quarterbalanced} could feasibly be interpreted as evidence towards such a phenomenon. Our next result rules out any such possibility in a strong sense, demonstrating that the answer to Question~\ref{question:mainquestion} depends only on whether $\eps\leq 1/4$ or not.
    \begin{theorem}\label{thm:anybalanced}
         Let $G$ be a $2$-coloured locally $\eps$-balanced $K_n$, and suppose $n \geq 2^{C t/\eps^{16}}$ for some absolute constant $C$. Then, $G$ contains a complete subgraph on $2t$ vertices where one colour class forms a $K_{t,t}$, or a complete subgraph isomorphic to a $\mathcal{P}_{t}$.
    \end{theorem}
    Obtaining a better quantitative dependence between $n$ and $\eps$ remains an intriguing challenge. We discuss this further in Section~\ref{sec:concludingremarks}. We remark also that the assertion in the abstract is a consequence of the above result.

    In \cite{BLM}, Theorem~\ref{thm:foxsudakov} was generalised to $r$-edge-colourings of $K_n$, and we refer the reader to that paper to see what kind of patterns can be guaranteed in the multicolour version of the problem. It is also natural to investigate a locally balanced version of the multicolour problem.  The situation here is unexpectedly more difficult, as already for locally $\eps$-balanced $3$-colourings there is no straightforward analogue of Theorem~\ref{thm:anybalanced}. We formalise this negative statement in Section~\ref{sec:morecolours}, and obtain some complementary positive results. 

   \par \textbf{Organisation of the paper.} In Section~\ref{sec:blow-ups}, we introduce a general tool to find patterns (such as those depicted in Figure~\ref{fig:3-colours}) efficiently in edge-coloured complete graphs. Our main tool is a technique of Nikiforov which uses subgraph count estimates to find large blow-ups of small subgraphs. This is in contrast to many of the related results in this area for which the main tool is the dependent random choice method. In Section~\ref{sec:overview}, we compare these two approaches and give an overview of the proofs of Theorems~\ref{thm:quarterbalanced} and Theorems~\ref{thm:anybalanced}. Section~\ref{sec:proofs} contains formal proofs of both of these results. In Section~\ref{sec:morecolours}, we treat the multicolour version of the problem.
    
    \section{Preliminaries and proof outline}\label{sec:prelims}
    \subsection{Notation}\label{sec:notation}
    In this subsection, in preparation to give an overview of the proofs, we introduce and recall some notation. An $r$-edge-coloured (or just an $r$-coloured) graph is a graph together with a labelled partition of its edge-set into $r$ parts. We view the parts of this partition as being labelled by different colours. We say that one coloured graph $G$ contains another coloured graph $H$ if there exists an injection $\phi\colon V(H)\to V(G)$ such that $(a,b)\in E(H)$ implies that $(\phi(a),\phi(b))\in E(G)$ and $(\phi(a),\phi(b))$ has the same colour as $(a,b)$.
    \par  We say that a coloured graph $G$ is a \textbf{homogeneous $t$-blow-up} of another coloured graph $H$ if $V(G)$ can be partitioned into $t$-sized sets $V_1,\cdots, V_{|V(H)|}$, each $G[V_i]$ is a monochromatic clique (in some colour), and for each $(i,j)\in E(H)$ of colour $c$, $G[V_i,V_j]$ is a complete bipartite graph where each edge gets colour $c$.
    \par An $r$\textbf{-totally-coloured} graph is a graph whose vertices as well as its edges are given an $r$-colouring. A coloured graph $G$ is a $t$\textbf{-blow-up of a totally-coloured graph} $H$ if $G$ is a $t$-blow-up of $H$ with respect to the edge-colouring of $H$, and $G[V_i]$ is a monochromatic clique in the same colour as the vertex $i\in H$. 
   
    \par  If $G$ is $2$-coloured, $\overline{G}$ is used to denote the coloured graph with the two colours interchanged.
    
    We say that an $r$-coloured $G$ is \textbf{locally $\eps$-balanced} if each vertex of $G$ is incident to at least $\eps|V(G)|$ edges in each of the $r$ colours. We say that an $r$-coloured graph $G$ is \textbf{globally $\eps$-balanced} if each colour class has size at least $\eps \binom{n}{2}$.

    \par For the convenience of the reader, in the following paragraph we collect every coloured and totally coloured graph we make reference to throughout the paper. For further clarity, in Figure~\ref{fig:allthegraphs}, we provide an illustration of each of these graphs.
    
    \par Let $P_1$ be the totally-coloured $K_2$ with vertices receiving colour red, and the unique edge receiving colour blue. Note that the $t$-blow-up of $P_1$ is a blue induced $K_{t,t}$. Let $P_2$ be the totally-coloured $K_2$ with one vertex receiving colour red, the other receiving colour blue, and the edge receiving colour red. Let $P_3$ be the totally-coloured $K_4$ with edges $(1,2),(2,3), (3,4)$ and vertices $2$ and $3$ coloured blue, and all other edges and vertices coloured red. Note that a $t$-blow-up of $P_3$ is locally $1/4$-balanced. Moreover, let $\ptm$ be the (not totally) coloured graph obtained from $P_3$ by discarding the vertex colours. Let $C_4$ be a $2$-edge-coloured $K_4$ with the red edges forming a $4$-cycle. Let $M_1$ be the properly $2$-edge-coloured $K_{2,2}$.


    With our notation, the aforementioned theorems can be stated as follows.
    
    \begin{reptheorem}{thm:foxsudakov}[Fox-Sudakov, Cutler-Mont\'agh]
    Suppose $n\geq(1/\eps)^{16t}$. Then, any globally $\eps$-balanced $2$-coloured $K_n$ contains a $t$-blow-up of one of $P_1$, $P_2$, $\overline{P_1}$, $\overline{P_2}$.
    \end{reptheorem}

    \begin{reptheorem}{thm:quarterbalanced}
         Suppose $n\geq 2^{t\cdot 2^{-C\log(1/\eps)^8}}$ for a sufficiently large absolute constant $C$. Then, any locally $(1/4+\eps)$-balanced $2$-coloured $K_n$ contains a $t$-blow-up of $P_1$ or $\overline{P_1}$.
    \end{reptheorem}
    
    \begin{reptheorem}{thm:anybalanced} 
       Suppose $n\geq 2^{C t/\eps^{16}}$ for a sufficiently large absolute constant $C$. Then, any locally $\eps$-balanced $2$-coloured $K_n$ contains a $t$-blow-up of $P_1$, $\overline{P_1}$, $P_3$.
    \end{reptheorem}
    Note that $\overline{P_3}$ is not included in the above list as $\overline{P_3}=P_3$.
 \subsection{Proof overview} \label{sec:overview}
 Although our proofs are short, the method is quite different from the other papers in the area. In this section, we aim to motivate our approach and give heuristic explanations for why Theorems~\ref{thm:quarterbalanced} and \ref{thm:anybalanced} are true, at least qualitatively (i.e. with some finite $n=n(t,\eps)$). We begin with a brief discussion of Theorem~\ref{thm:foxsudakov} in order to compare the different approaches. It is rather straightforward to obtain a proof of Theorem~\ref{thm:foxsudakov} if one is only concerned with obtaining $t$-blow-ups where $t$ tends to infinity with $n$. Indeed, applying the well-known K\H{o}vari-S\'os-Tur\'an Theorem to both red and blue colour classes, we can obtain one red and one blue $K_{s,s}$, say $T_1$ and $T_2$, respectively, where $s\sim \log n$ (we can ensure that $T_1$ and $T_2$ are vertex-disjoint). Apply Ramsey's theorem to the four parts coming from $T_1$ and $T_2$, and delete all vertices outside of the monochromatic cliques guaranteed by this application. Similarly, for each pair of the $4$ parts, in turn, apply the K\H{o}vari-S\'os-Tur\'an theorem to the majority colour class in the bipartite graph between the pair of parts to guarantee a monochromatic complete bipartite subgraph, at each iteration reducing the size of the parts logarithmically. This produces a $t$-blow-up of a totally-coloured $K_4$ (where $t\to \infty$ as $n\to\infty$) where we do not have precise control over the colouring; however, we know that the colouring is not monochromatic (there exists at least one blue and one red edge). It is easy to see that this implies that the totally-coloured $K_4$ must contain one of $P_1$, $P_2$, $\overline{P_1}$, $\overline{P_2}$, as desired.
 
 \par In \cite{FS}, the dependent random choice method (see, e.g.~\cite{fs11}) is employed to obtain optimal bounds for Theorem~\ref{thm:foxsudakov}, eliminating the need for nested applications of the Theorems of Ramsey and K\H{o}vari-S\'os-Tur\'an. However, the argument still boils down to finding a blow-up of a totally-coloured graph which is not monochromatic, thereby producing a structure which necessarily contains one of the desired patterns. To prove Theorems~\ref{thm:quarterbalanced} and ~\ref{thm:anybalanced}, we find blow-ups of totally-coloured graphs where each vertex is adjacent to edges of both colours. The argument given in the previous paragraph is too weak to achieve this, since it does not make use of the assumption that our colouring is locally-balanced. The  dependent random choice is quantitatively much stronger, but does not give any additional structural information. 
 
 \par Hence, we need an argument which will make use of the global structure of the colouring. Specifically, our main tool (Lemma~\ref{l:r-partite-coloured}), a strengthening of a theorem due to Nikiforov~\cite{nikiforov08}, reduces the problem to finding certain small subgraphs appearing in $G$ with a positive density. Using Lemma~\ref{l:r-partite-coloured}, the statements of Theorem~\ref{thm:quarterbalanced} and \ref{thm:anybalanced} reduce to the following two statements, respectively.
     \begin{proposition} \label{prop:many-c4} There exists a constant $C_1$ such that the following holds.
        Any locally $(1/4 + \eps)$-balanced 2-coloured $K_n$ contains $2^{-C_1\log(1/\eps)^8} n^4$ copies of $C_4$ or $\overline{C_4}$.
    \end{proposition}
    \begin{proposition} \label{prop:many-p3c4}
    Any locally $\eps$-balanced $2$-coloured $K_n$ contains  $\eps^4n^4/10^5$ copies of $\ptm, C_4$ or $\overline{C_4}$.
    \end{proposition}

    A natural tool for proving statements such as Proposition~\ref{prop:many-c4} and \ref{prop:many-p3c4}  is the Szemer\'edi Regularity Lemma and closely related removal lemmas (see, e.g.,~\cite{cf13}). For the sake of obtaining better bounds and more cogent proofs, we do not use the regularity method. However, the regularity method does provide a shorter proof of a quantitatively weaker version of Proposition~\ref{prop:many-c4}, and thus Theorem~\ref{thm:quarterbalanced}. Indeed, using the Regularity Lemma, it is not hard to show that the vertex set of a $2$-coloured $K_n$ with a vanishing density of both $C_4$ and $\overline{C}_4$ can essentially (that is, after recolouring $o(n^2)$ edges)), be partitioned into a red and a blue clique. Such a $2$-coloured $K_n$ can at best be locally $(1/4+o(1))$-balanced, by an easy optimisation argument (see Proposition~\ref{prop:optimize}), implying Proposition~\ref{prop:many-p3c4}.

    Proposition~\ref{prop:many-p3c4} is, however, more subtle, since even finding a single copy of the desired small subgraphs is a non-obvious extremal problem. To give some intuition on why the statement holds, let us prove an idealised version of it.
    \begin{proposition}\label{prop:cute}
 Let $(A,B)$ be a $2$-coloured complete bipartite graph with $|A|,|B|>2$ so that each vertex of $A$ has at least one blue neighbour, and each vertex of $B$ has at least one red neighbour. Then, $(A,B)$ contains a cycle of length $4$ whose edges are alternating red and blue.
 \end{proposition}
 \begin{proof}
     Suppose without loss of generality that $|A|>|B|$. Let $v,w\in A$. Suppose there exists no alternating red and blue $4$-cycle. Observe that the blue neighbourhood of $v$ must be contained in that of $w$, or vice versa, otherwise we immediately obtain an alternating $4$-cycle. Hence, the collection of blue neighbourhoods of vertices in $A$ forms a chain where the smallest subset has at least one element, say $q\in B$, by assumption. Hence, $q$ is in the blue neighbourhood of every vertex of $A$, meaning $q$ has no red neighbours, a contradiction. 
 \end{proof}
    
    To clarify the connection between Proposition~\ref{prop:cute} and Proposition~\ref{prop:many-p3c4}, notice that completing a colouring of an \textit{alternating} red-blue 4-cycle to an edge-colouring of $K_4$ yields a $C_4, \overline{C_4}$ or a $\ptm$.

    Finally, we state our main technical tool, which allows us to deduce Theorems~\ref{thm:quarterbalanced} and ~\ref{thm:anybalanced} from the aforementioned Propositions. Nikiforov~\cite{nikiforov08} showed that any graph with a positive density of $K_\ell$-copies contains a \textit{large} complete $\ell$-partite subgraph. We strengthen this statement to find \textit{homogeneous} blow-ups in $r$-coloured graphs. Recall that our definition of a homogeneous blow-up in a coloured graph requires that the parts of the blow-up induce monochromatic cliques, but the colours of the cliques are not specified.
   
    \begin{lemma} \label{l:r-partite-coloured} 
    	Let $H$ be an $\ell$-vertex $r$-coloured graph. Let $G$ be an $n$-vertex $r$-coloured graph $G$ containing at least $c n^\ell$ copies of $H$. Then  $G$ contains a homogeneous $t$-blow-up of $H$ with $ t \geq \min\left\{ \frac{c}{2\ell}, \frac{1}{2 r \log r} \right\} ^\ell \log n$.
\end{lemma}
    Lemma~\ref{l:r-partite-coloured} is proved in Section~\ref{sec:blow-ups}. We remark that the Lemma also gives a short proof of Theorem~\ref{thm:foxsudakov} with the stronger hypothesis that $n\geq 2^{100t/\eps}$. Indeed, an easy counting argument  implies that any globally $\eps$-balanced 2-colouring of $K_n$ contains  $\Omega(\eps n^3)$ properly coloured two-edge paths (that is, $K_{1, 2}$-copies), so  Lemma~\ref{l:r-partite-coloured} yields a homogeneous $t$-blow-up of the non-monochromatic edge-colouring of $K_{1,2}$. A simple case distinction then gives a blow-up of $P_1$, $P_2$ or one of their complements. A similar argument gives a concise proof (with a weaker dependence between $n$ and $\eps$) of Theorem 1.4 from \cite{BLM} , which is a generalisation of Theorem~\ref{thm:foxsudakov} to an arbitrary number of colours, originally proved via nested applications of the dependent random choice method.
    \section{Proofs for patterns in two-colourings} \label{sec:proofs}
    \subsection{Proofs of the main theorems}
        We now show how the main theorems follow easily from the two propositions on subgraph counts and Lemma~\ref{l:r-partite-coloured}.
    \begin{proof}[Proof of Theorem~\ref{thm:quarterbalanced}]
        By Proposition~\ref{prop:many-c4}, we may assume (without loss of generality) that $G$ contains at least $2^{-C_1\log(1/\epsilon)^8} n^4$ copies of $C_4$. By Lemma~\ref{l:r-partite-coloured} applied with $c = 2^{-C_1\log(1/\epsilon)^8}$, $G$ contains a homogeneous $2^{-C\log(1/\epsilon)^8} \log n)$-blow-up of $C_4$ for some absolute constant $C$. Note that by assumption on $n$, $2^{-C\log(1/\epsilon)^8} \log n \geq t$. Let $V_1, \ldots, V_4$ be the parts of this blow-up, and recall that $G[V_i]$ are monochromatic cliques.

        We claim that $G[V_1 \cup V_2 \cup V_3 \cup V_4]$ contains a $t$-blow-up of $P_1$ or $\overline{P_1}$. Assume the opposite. At least one vertex part has to be blue, so assume that $G[V_1]$ is blue. Then $G[V_2]$ and $G[V_4]$ are red, so they form a blow-up of $P_1$, which is a contradiction.
     \end{proof}

         \begin{proof}[Proof of Theorem~\ref{thm:anybalanced}] 
        Proposition~\ref{prop:many-p3c4} implies that $G$ contains $\eps^4 n^4/10^5$ copies of $C_4$, $\overline{C_4}$ or $\ptm$ in $G$. Let $t = \left(\eps^4 \cdot 10^{-6} \right)^4 \log n$, noting that this satisfies the claimed bound on $n(t, \eps)$. In the cases with many copies of $C_4$ or $\overline{C_4}$, the above argument from the proof of Theorem~\ref{thm:quarterbalanced} yields a $t$-blow-up of $P_1$ or $\overline{P_1}$. In the latter case, applying Lemma~\ref{l:r-partite-coloured}, we obtain that $G$ contains a homogeneous $(\eps^{16}10^{-7}\log n)$-blow-up of $\ptm$ on the vertex parts $V_1, V_2, V_3$ and $V_4$. Suppose that this structure does not contain the $t$-blow-up of a $P_1$ or $\overline{P}_1$, otherwise we are done. Suppose that $G[V_1]$ is blue. Then $G[V_3]$ and $G[V_4]$ are red, and hence they form a blow-up of $P_1$, contradiction. Hence $G[V_1]$ is red, and consequently, $G[V_2]$ is blue, so $G[V_4]$ is red, and finally, $G[V_3]$ is blue. This yields a $t$-blow-up of $P_3$, as required.
     \end{proof}

\subsection{Finding homogeneous blow-ups of small patterns} \label{sec:blow-ups}
 In this subsection, we prove Lemma~\ref{l:r-partite-coloured}, a variant of a  result due to Nikiforov \cite{nikiforov08}. The main ingredient is Lemma~\ref{l:r-partite-hyp}, which is about dense $\ell$-uniform hypergraphs. Before stating and proving Lemma~\ref{l:r-partite-hyp}, we give the following definitions which are central due to the inductive nature of the proof.  
 
 Let $\Hc$ be an $\ell$-partite $\ell$-uniform hypergraph $\Hc$ on parts $V_1, \ldots, V_\ell$. We usually specify the members of $E(\Hc)$ as $\ell$-tuples $(v_1, \ldots v_\ell)$ (where $v_i\in V_i$), but we also abuse notation by saying that $(v_1, \ldots, v_\ell)$ contains $v_1$,  or by writing $(v_1, \ldots v_\ell) = (v_1, \ldots, v_{\ell-1})+v_\ell$. Let $K_2(\Hc)$ be the vertex pairs contained in edges of $\Hc$. Moreover, we write $\partial \Hc$ for the collection of $(\ell-1)$-tuples $(v_1, \ldots, v_{\ell-1})$ which are contained in some edge $(v_1, \ldots, v_\ell)$ of $\Hc$, with $v_i \in V_i$ for $i \in [\ell]$. We emphasise that this notation is not standard since $\partial \Hc$ only contains $(\ell-1)$-tuples from $V_1 \times \dots \times V_{\ell-1}$. Given a graph $G$ isomorphic to a $K_{\ell}(m)$ (the complete $\ell$-partite graph with vertex parts of size $m$), we say that $\Hc$ \emph{covers} $G$ if $E(G) \subset K_2(\Hc)$ and there are $m$ disjoint $S \in \Hc$ with $S \subset V(G)$ (in other words, $\Hc$ contains a matching of size $m$ on $V(G)$). 

\begin{lemma} \label{l:r-partite-hyp}
	Let $\Hc$ be an $\ell$-partite $\ell$-uniform hypergraph on parts $V_1, \ldots, V_\ell$  of size $n$ with least $\ell c n^\ell$ edges. Let $\varphi$ be an $r$-colouring of the complete ($2$-uniform) graphs on $V_1, \ldots, V_\ell$. Then, there are vertex sets $S_1, \ldots, S_\ell$  with $|S_i| \geq \left( \min\{\frac c2, \frac {1}{2r \log r}\} \right)^\ell \log n$ such that $\varphi$ is constant on each $S_i$, and $\Hc$ covers the complete $\ell$-partite graph on $S_1, \ldots, S_\ell$.
\end{lemma}

To prove Lemma~\ref{l:r-partite-hyp}, we need the following lemma for finding unbalanced complete bipartite graphs, which can be proved using the classical double counting argument of K\H{o}vari, S\'os and Tur\'an~\cite{kst54}.

\begin{lemma}[Lemma 2 from~\cite{nikiforov08}] ~\label{l:stars}
	Let $F$ be a bipartite graph with parts $A$ and $B$. Let $|A| = m$, $|B| = n$ and $c, \alpha \in (0, 1/2)$. If $\alpha \log n \leq \frac {cm}{2} + 1$ and $e(F) \geq cmn$, then $F$ contains a complete bipartite graph with parts $S \subset A$ and $T \subset B$ of size $|S| = \alpha \log n$ and $|T| > n^{1-\alpha/c}$.
\end{lemma}
We also need a standard upper bound on the $r$-colour Ramsey number due to Erd\H{o}s and Szekeres~\cite{es35}. Namely, any $r$-colouring of $K_n$ contains a monochromatic clique on $\frac{\log n}{r \log r}$ vertices. We emphasise that we are mostly concerned with the case $r =2$. We can now proceed with the proof.
\begin{proof}[Proof of Lemma~\ref{l:r-partite-hyp}]
	Assume that $c \leq (r \log r)^{-1}$, since the statement for larger values of $c$ then follows. We prove the statement by induction on $\ell$. The case $\ell=1$ follows from Ramsey's Theorem -- a colouring of $K_{cn}$ contains a monochromatic clique of size  $ (r \log r)^{-1} \log (cn) \geq (2 r \log r)^{-1} \log (n)  \geq \frac  c2 \log n$. Assume that the statement holds for $\ell-1$, and let $\Hc$ be as in the statement.
	
	For a hypergraph $\Lc$ and $R= (v_1, \ldots, v_{\ell-1})$, we write $d_{\Lc}(R)$ for the number of edges of $\Lc$ containing $R$. A standard deletion argument shows that there is $\Lc \subset \Hc$ with $|\Lc| \geq (\ell-1)cn^\ell$ such that 
	\begin{equation} \label{eq:min-deg}
	    \text{for any } R \in V_1 \times \dots \times V_{\ell-1}, \quad d_{\Lc}(R) \geq cn \text{\ \ or \ \ } d_{\Lc}(R)=0;
	\end{equation} indeed, one can iteratively remove all edges containing $R$ for any $R$ violating~\eqref{eq:min-deg}, removing at most $cn \cdot n^{\ell-1}$ edges in total.
	
	We have $|\partial \Lc| \geq |\Lc| / n \geq (\ell-1) cn^{\ell-1}$, since each $R \in \partial \Lc$ is contained in at most $n$ edges of $\Lc$. Applying the induction hypothesis to $\partial \Lc$, we obtain sets $U_1, \ldots U_{\ell-1}\subseteq V_1,\ldots, V_{\ell-1}$ with $|U_i| = m = \left(\frac c2 \right)^{\ell-1} \log n$ such that $\partial \Lc$ covers the complete $(\ell-1)$-partite graph on $U_1 \cup \ldots \cup U_{\ell-1}$, and $\varphi[U_i]$ is constant for each $i\in[\ell-1]$.
	
	Let $\Ac$ be a set of $m$ disjoint $(\ell-1)$-tuples in $\partial \Lc$, which exists by the definition of a \emph{covering}. The graph structure of $K_2(\Lc)$ will now be used by noticing that it suffices to find a \emph{large} subset of vertices $T \subset V_\ell$ such that $R+v_\ell \in \Lc$ for all $R \in \Ac$ and $v_\ell \in T$. To this end, consider the bipartite graph $F$ with parts $\Ac$ and $V_\ell$ such that $(R, v_\ell) \in F$ whenever $R+ v_\ell$ lies in $\Lc$. Since $d_{\Lc}(R) > cn$ for all $R \in \Ac \subset \partial \Lc$, we have $|F| \geq cmn$.
	
	We can apply Lemma~\ref{l:stars} with $s = \left(\frac c2 \right)^\ell \log n \leq \frac c2 m + 1$ and $t = n^{1-2^{-\ell}c^{\ell-1}}$. It follows that $F$ contains a complete bipartite graph with parts $\Ac'$ and $T$ such that $|\Ac'| = s$ and $|T| = t$. Let $G = K_2(\Lc)$. Let $S_1, \ldots, S_{\ell-1}$ be the vertex sets of the edges of $\Ac'$, and recall that they induce a $K_{\ell-1}(s) $ in $ G$ since $S_i \subset U_i$. Moreover, we claim that $w v_\ell$ is in $G$ for any $v_\ell \in T$ and $w \in S_i$ with $i \in [\ell-1]$. This follows from the fact that there is an $R \in \Ac'$ containing $w$, and $R+v_\ell \in \Lc$.
	
	Finally, by Ramsey's theorem, there is a subset $S_\ell \subset T$ with $|S_\ell| = s = \left(\frac c2 \right)^\ell \log n \leq \frac {1}{r\log r} \log |T|$ on which $\varphi$ is constant. Recalling that $\varphi[S_i]$ is constant for $i \in [\ell-1]$ by induction hypothesis, we obtain our desired sets
	$S_1, \ldots, S_\ell$. To verify that $\Lc$ covers $S_1, \ldots, S_\ell$, note that the edges of $\Ac'$ with the vertices of $S_\ell$ (in an arbitrary order) form a matching of size $s$ in $\Lc$. 
\end{proof}
We now give the proof of Lemma~\ref{l:r-partite-coloured}.
\begin{proof}[Proof that Lemma~\ref{l:r-partite-coloured} follows from Lemma~\ref{l:r-partite-hyp}]
	 Let $V_1, \ldots, V_\ell$ be a uniformly random partition of the vertex set of $G$ with parts of size at least $n' = \lfloor \frac{n}{\ell} \rfloor$, and associate $V_1, \ldots, V_{\ell}$ to the vertices of $v_1,v_2,\ldots, v_\ell$ of $H$. We say that a copy of $H$ in $G$ is canonical with respect to this partition if $v_i$ is embedded to $V_i$ for each $i\in[\ell]$. We claim that a copy of $H$ in $G$ is canonical with probability at least $(n' / n)^{\ell}$. Indeed, each vertex $v_i$ of this $H$-copy is placed into $V_i$ with probability at least $\frac{n'}{n}$. Moreover, the events that the vertices of this $H$-copy are placed into the corresponding parts are positively correlated, which implies the lower bound $(n' / n)^{\ell}$.
	 
	 Hence, using linearity of expectation, we may assume that the number of $H$-copies respecting our partition is at least  $\left(\frac {n'}{n} \right)^\ell cn^{\ell}= cn'^\ell$. Let $\Hc$ be the hypergraph corresponding to those copies. We may apply Lemma~\ref{l:r-partite-hyp} with $c' = \frac c\ell$, to obtain the desired sets $S_1, \ldots S_\ell$ with $|S_i| =\min\left\{ \frac{c}{2\ell}, \frac{1}{2 r \log r} \right\} ^\ell  \log n$.
\end{proof}

    \subsection{Small subgraphs in locally balanced colourings}
    We now prove Proposition~\ref{prop:many-c4}, which follows immediately from Lemma~\ref{lem:maintech} and Proposition~\ref{prop:optimize}. As discussed above, in Lemma~\ref{lem:maintech} we actually describe the structure of colourings with a vanishing density of $C_4$ and $\overline{C_4}$. 
    We call a $2$-coloured $K_n$ \textbf{split} if its vertex set can be partitioned into a red clique and a blue clique. We call two $2$-coloured $K_n$ $\delta$-\textbf{close} if the first can be made isomorphic to the second after flipping the colours of at most $\delta n^2$ many edges. The following result can be thought of as a substitute for the regularity method which is sufficient for our purposes.
        \begin{lemma}[Fox-Sudakov, \cite{fox2008induced} (Theorem 4.4)]\label{lem:partitionintohombits} There exists an absolute constant $c$ such that for each $\eps\in (0,1)$ and graph $H$ on $k$ vertices, there are constants $\kappa:=(\eps/4)^k2^{-c(k\log(1/\eps))^2}$ and $C=4/(\eps 2^{-ck(\log(1/\eps))^2})$ such the following holds. For any $n$-vertex graph $G$ with fewer than $\kappa n^k$ induced copies of $H$, there is an equitable partition of $V(G)$ into at most $C$ parts such that each part induces a subgraph of density at most $\eps$ or at least $1-\eps$.
    \end{lemma}
    We now use the above lemma to show that a $2$-coloured graph with a vanishing density of $C_4$ and $\overline{C_4}$ is $o(1)$-close to being split.
    \begin{lemma}\label{lem:maintech}
    There exists an absolute constant $C_2$ so that the following holds for any $n$ and $\delta$. Let $G$ be a $2$-coloured $K_n$. Then, at least one of the following is true.
    \begin{enumerate}
        \item $G$ contains at least $2^{-C_2\log(1/\delta)^8} n^4$ many distinct copies of $C_4$ or $\overline{C_4}$
        \item $G$ is $\delta$-close to being split.
    \end{enumerate}
    \end{lemma}
    \begin{proof}
        Suppose that $(1)$ does not hold. Then, by Lemma~\ref{lem:partitionintohombits} applied with $\delta^{5}$ in place of $\eps$, we have that there is an equitable partition of $G$ into at most $$C=4/(\delta^5 2^{-4c(\log(1/\delta^5))^2})$$ parts, each of which has either red or blue density above $1-\delta^5$. Let us refer to the parts with high red density as \textit{red parts}, and label them by $V_1,\ldots, V_\ell$.

        We claim that $G[V_1 \cup \dots \cup V_\ell]$ contains at most $\delta n^2 /2$ blue edges. If $\ell = 1$, this is trivial, and otherwise it follows from the following claim.
        \begin{claim}
             Let $V_i$ and $V_j$ be two parts with density $\geq 1-\delta^5$ in red. Then the bipartite graph $G[V_i, V_j]$  has red density at least $1-\delta$.
        \end{claim}
        \begin{proof}
            Suppose that the blue density between $V_i$ and $V_j$ is at least $\delta$. It follows that the blue subgraph of $G[V_i\cup V_j]$ must contain at least $2^{-10} \delta^4 (n/C)^4$ copies of a cycle on $4$ vertices (see, for example, Theorem 1.9(iv) from \cite{nagy2019supersaturation}). At most $2\delta^5 (n/C)^4$ such $4$-cycles can contain a blue edge from $G[V_i]$  or $G[V_j]$, by the density assumption. The remaining $4$-cycles must correspond to a copy of $\overline{C_4}$ in $G[V_i\cup V_j]$. Note that for some absolute constant $C_2$, we have that $\delta^4 (n/C)^4\geq2^{-C_2\log(1/\delta)^8} n^4$, so as we assumed that $(1)$ does not hold, we conclude that $G[V_i, V_j]$ has red density at least $1-\delta$.
        \end{proof}
        The same argument implies that the union of blue parts contains at most $\delta n^2 / 2$ blue edges. Hence, combining the red parts as well as the blue parts gives a partition certifying that $G$ is $\delta$-close to a split graph, as required.
    \end{proof}
  We now show that graphs which are $\delta$-close to being split cannot be $(1/4+2\delta)$-balanced. When combined with Lemma~\ref{lem:maintech}, this immediately implies Proposition~\ref{prop:many-c4}.
    \begin{proposition}\label{prop:optimize}
    If a $2$-coloured $K_n$ is $\delta$-close to being split, then it has a vertex with at most $(1/4+3\delta)n$ red or at most $(1/4+3\delta)n$ blue neighbours.
    \end{proposition}
    \begin{proof}
        Assume for the sake of a contradiction that $G$ is a $2$-coloured $K_n$ in which all vertices have more than $(1/4+3\delta)n$ red and more than $(1/4+3\delta)n$ blue neighbours. Consider a split graph $G'$ which is $\delta$-close to $G$. The vertex set of $G'$ is the union of a red clique $X$ and a blue clique $Y$. In $G$, the sum of the blue degrees of the vertices in $X$ is $>(1/4+3\delta)n|X|$. Since $X$ contains at most $\delta n^2$ blue edges by the $\delta$-closeness assumption between $G$ and $G'$, the edges inside $X$ contribute at most $2\delta n^2$ to the previous sum. This implies that in $G$, between $X$ and $Y$, there are at least $(1/4+\delta)n|X|$ blue edges. Similarly, we can derive that between $X$ and $Y$, there are at least $(1/4+\delta)n|Y|$ red edges, giving in total $(1/4+\delta)n^2>n^2/4$ edges between two disjoint subsets of $G$. This is a contradiction.
    \end{proof}
   Now we move on to Proposition~\ref{prop:many-p3c4}. Recall that $M_1$ denotes a properly $2$-edge-coloured $K_{2,2}$. The following lemma is a robust version of Proposition~\ref{prop:cute}. For a graph $G$, let $\delta_G(S)$ denote the minimum degree of a vertex (in $G$) in $S \subset V(G)$.
    
    \begin{lemma}\label{lem:alternatingmatchings} Let $\eps>0$, and consider a two-colouring $R \cup B$ of $K_{n,n}$ with vertex parts $X, Y$ of order $n$ such that $\delta_R(X), \delta_B(Y)\geq \eps n$. Then, this two-colouring of $K_{n,n}$ contains at least $\eps^4 n^4/150$ many distinct copies of $M_1$. 
    \end{lemma}
    \begin{proof}
    Let $A$ be the set of vertices in $X$ contained in at least $\eps^3 n^3/72$ copies of $M_1$. We will show that $| A | \geq \frac{\eps n }{2}$. This implies the lemma, since the number of copies of $M_1$ is at least $|A|\eps^3 n^4 / 150$. 

    Assume that $|A | < \frac{\eps n }{2}$, and let $A'  = X \setminus A$. Note that $\delta_R(A'), \delta_B(Y)\geq \eps n/2$ by the assumption on the size of $A'$. Let $v$ be a vertex of $A'$ with minimum red degree, and let $S\subseteq Y$ be the red neighbourhood of $v$. In $A'\setminus\{v\}\times S$, there are at least $\eps n^2/3$ blue edges. It is well-known that every graph with density $\mu$ contains a subgraph with minimum degree $\mu n/2$. So, $A'\setminus\{v\}\times S$ has a subgraph $(C,D)$ with minimum blue degree at least $\eps n/6$ (in particular, $|C|, |D|\geq \eps n/6$). Observe now that for each element $d$ of $D$, each blue neighbour of $d$ in $C$, say $c$, and each red neighbour of $c$ in $Y\setminus S$, say $b$, $\{v,d,c,b\}$ induce a copy of $M_1$. This is because $(v,b)$ must be blue by the assumption that $b\notin S$. There are $|D|\eps n / 6$ choices for such $d$ and $c$, and we claim that any $c \in C$ has at least $\eps n / 6$ red neighbours in $Y \setminus S$, all of which can play the role of $b$. To see this, note that by minimality of $v$, $c$ must have at least $|S|$ red neighbours in $Y$, and that at most $|S|-\eps n/6$ of these neighbours are contained in $S$ (since $c$ has at least $\eps n/6$ blue neighbours in $D \subset S$). Hence, we can find at least $\eps^3 n^3/72$ distinct triples $(d,c,b)$ giving rise to distinct $M_1$ containing $v$. This contradicts the definition of $A$, so we conclude that $|A| \geq \eps n/2$.
    \end{proof}
    
    Proposition~\ref{prop:many-p3c4} follows easily from the previous result.
   
    \begin{proof}[Proof of Proposition~\ref{prop:many-p3c4}]
    Let $G$ be locally $\eps$-balanced 2-coloured $K_n$. By a standard random sampling argument, we can deduce that $G$ contains some $G'=(X, Y)$ which satisfies the hypotheses of Lemma~\ref{lem:alternatingmatchings} with $n:=n/2$ and $\eps:=\eps/2$ (for example, Chernoff's bound is sufficient here). It follows by Lemma~\ref{lem:alternatingmatchings} that $G$ contains $\eps^4 n^4/10^5$ many distinct copies of $M_1$. Note that each edge is present in $M_1$, so each copy of $M_1$ yields at least one copy of $P_3, \overline{P_3}, C_4$ or $\overline{C_4}$ depending on the colour of the edges not given by the $M_1$. This implies the proposition.
    \end{proof}

    \section{Multiple colours}\label{sec:morecolours}
    
    In this section, we investigate the $r$-colour variant of Question~\ref{question:mainquestion}. To allow for a more precise discussion, we give the following two definitions. Given a totally-coloured graph $H$, the $t$-blow-up of $H$ is denoted by $H[t]$. A totally $r$-coloured graph $H$ is called \textbf{unibalanced} if each vertex of $H[2]$ is incident to an edge in each of the $r$ colours.\footnote{In other words, $H$ is unibalanced if every vertex $v$ of $H$ is incident to each of the available $r$ colours (the value of $r$ will be clear from context), where a vertex $v$ is also considered incident to $c$ if it is coloured $c$.} For instance, the patterns $P_1$ and $P_3$ are unibalanced, but $P_2$ is not. Observe that for a totally-coloured graph $H$, $H[t]$ is locally $\eps$-balanced for some value of $\eps>0$ if and only if $H$ is unibalanced. We call a family $\mathcal{F}$ of $r$-colourings of $K_n$ locally $(r,\eps)$-\textbf{unavoidable} if every locally $\eps$-balanced colouring of $K_n$ where $n$ is sufficiently large contains a copy of some $F\in \mathcal{F}$. 
 \begin{question} \label{que:multicolour} Suppose $K_n$ is given a locally $\eps$-balanced $r$-edge-colouring. Which subgraphs must $K_n$ necessarily contain? 
 \end{question}
 Obviously, colourings as in Question~\ref{que:multicolour} are globally $\eps$-balanced as well. In \cite{BLM}, a multicolour version of Theorem~\ref{thm:foxsudakov} is provided (see Theorem 4). This result guarantees, just using a global balancedness assumption, the existence of a rich family of patterns, say $\mathcal{F}^{(r)}$, where each pattern in the family uses all $r$ colours (see Figure 1 in \cite{BLM} for a depiction of patterns guaranteed when $r=3$). In particular, $\mathcal{F}^{(r)}$ is locally $(r,\eps)$-unavoidable for every $\eps>0$. Yet, $\mathcal{F}^{(r)}$ does not give a satisfying answer to Question~\ref{question:mainquestion}, for the same reason the family $\mathcal{F}_t$ (see Figure~\ref{fig:unavoidable2patterns}) does not give a satisfying answer to Question~\ref{question:mainquestion}. That is, elements of $\mathcal{F}^{(r)}$ are not locally balanced, leaving open the possibility that there exists either a smaller (as in Theorem~\ref{thm:quarterbalanced}) or more complex (as in Theorem~\ref{thm:anybalanced}) family of unavoidable patterns. On the other hand, as in the two-colour case, any $(r,\eps)$-unavoidable family $\mathcal{F}$ consisting of blow-ups of a finite family of unibalanced totally-coloured graphs $\mathcal{F}'$ (assuming that no member of $\mathcal{F}'$ contains another member) would give a structurally optimal answer to Question~\ref{question:mainquestion}). Indeed, members of such $\mathcal{F}$ would be themselves locally $\eps$-balanced for some $\eps>0$, making it impossible for a smaller, or more complex family of $(r,\eps)$-unavoidable graphs to exist (for $\eps$ sufficiently small).

We can now explain why already for $r\geq 3$, the problem is rather different. A consequence of Theorem~\ref{thm:anybalanced} is that for \textit{any} $\eps>0$ and $t\geq 1$, the family $\{P_1[t],\bar{P_1}[t], P_3[t]\}$ is $(2,\eps)$-unavoidable.  For $r= 3$, if $\Fc$ is a finite family consisting of unibalanced totally-coloured graphs whose blow-ups are $(3,\eps)$-unavoidable, then the following proposition shows that $|\Fc|$ has to have size at least $1/\eps$.

\begin{figure}[h!bt]
\centering
\includegraphics[width=0.8\textwidth]{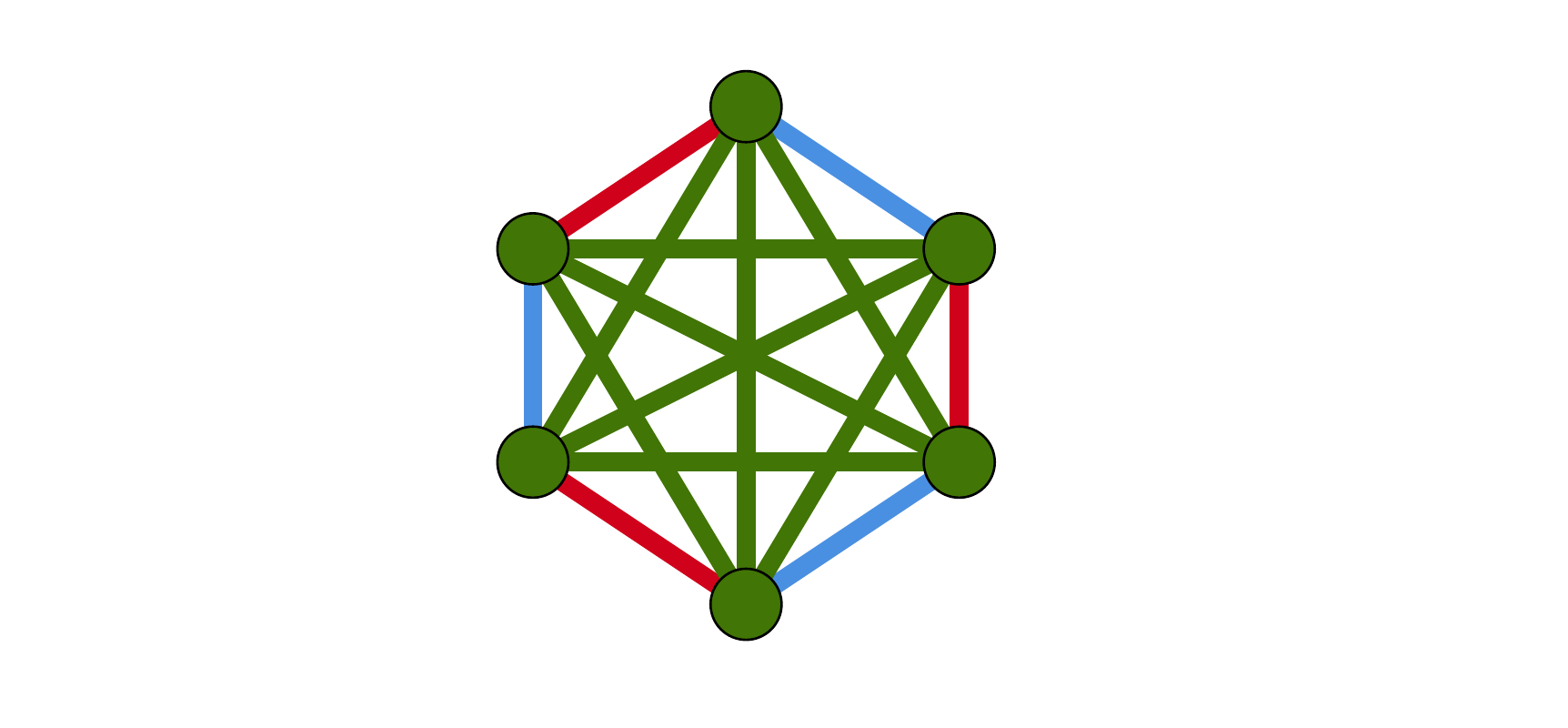}

    \caption{A three-colouring of $K_n$ in which any locally balanced subgraph has at least $6$ vertices, where $6$ is the number of parts. This colouring is locally $(1/6)$-balanced. Colourings of this type (where two colour classes form a blow-up of an alternating cycle) show that there is no analogue of Theorem~\ref{thm:anybalanced} for more than two colours.}
    
    \label{fig:3-colours}
\end{figure}


    \begin{proposition} ~\label{prop:3colourfail}
        There exists a locally $\eps$-balanced $3$-colouring of $K_n$ in which any unibalanced subgraph has at least $\lfloor \eps^{-1} \rfloor-1$ vertices.
    \end{proposition}
    \begin{proof}
        Partition the vertex set into parts $V_1, \ldots, V_\ell$ of order at least $\eps n$, where $\ell \geq \lfloor \eps^{-1} \rfloor-1$ is an even number. Colour the bipartite graphs between $V_i$ and $V_{i+1}$ red and blue alternatingly (with indices modulo $\ell$). Colour the remaining edges green. For an illustration of the case when $\ell=6$, see Figure~\ref{fig:3-colours}. This colouring is locally $\eps$-balanced and any non-empty unibalanced subgraph of this colouring must contain a vertex from each of $V_1, \ldots, V_\ell$, as required.
    \end{proof}

Proposition~\ref{prop:3colourfail} implies that there can be no straightforward generalisation of Theorem~\ref{thm:anybalanced} for more than two colours. That is, there is no single finite family of unibalanced totally-coloured (with $r$ colours) graphs whose blow-ups are $(r,\eps)$-unavoidable for every $\eps>0$. However, we can still prove a version of Theorem~\ref{thm:anybalanced}, with the caveat that the size of the family of unavoidable patterns depends on $\eps$.

\begin{theorem} \label{t:multicol}
	Given $\gammep >0$, an integer $r$, and $C = \frac{80}{\gammep}  \log  \frac 1 \gammep  $, there is a constant $\alpha$ such that the following holds for sufficiently large $n$. Any locally $\gammep$-balanced $r$-colouring of an $n$-vertex graph $G$ contains a homogeneous $\alpha \log n$-blow-up of some unibalanced graph on at most $C$ vertices.
\end{theorem}
\begin{proof}
	Let $\gammep \gg \eta \gg \alpha \gg 1/n$. We have the following claim.
	\begin{claim} \label{c:random-subset}
		There exists $k \leq C$ such that there are at least $\frac 14 \binom nk$ vertex subsets $S$ of order $k$ inducing a unibalanced $r$-coloured subgraph.
	\end{claim}
	\begin{proof}
		Denote the family of subsets $S \subset [n]$ which induce a unibalanced subgraph by $\mathcal{U}$. Let $S$ be a random set of vertices where each vertex is sampled  independently with probability $\frac{\zeta}{n}$, with $\zeta = \frac{20}{\gammep} \left(\log r+ \log \frac{1}{\gammep} \right) \leq \frac C2$. We will show that the probability that $S \notin \mathcal{U}$ is at most $\frac 12$.

		For $v \in V(G)$, let $A_i(v)$ be the event that there is a vertex $u \in S$ such that $uv$ has colour $i$ (note that this event does not depend on whether $v$ is in $S$), and let $A(v) = \bigcap_{i \in [r]} A_i(v)$.
		
		Since $A_i(v)$ are mutually independent for $i \in [r]$, we have
		$$\pr{A(v)} = \prod_{i \in [r]}  \pr{A_i(v)}.$$
		
		Moreover, by the locally $\gammep$-balancedness assumption, we have that $\pr{A_i(v)} \geq 1- \left( 1- \frac{\zeta} {n}\right)^{\gammep n} \geq 1- e^{-\zeta \gammep /2}$ for $n$ sufficiently large, so 
		$$\pr{A(v)} \geq 1- re^{-\zeta \gammep /2}.$$
		Since $A(v)$ and $v \in S$ are independent events, we have, for any $v \in [n]$,
		$$\pr{\overline{A(v)} \land v \in S} \leq \frac{\zeta}{n}\cdot re^{-\zeta \gammep /2}$$
		
		Denoting the random variable counting the vertices $v \in S$ for which $A(v)$ does not occur by $X$ and substituting for $\zeta$, we have
		\begin{align*}
			\er{X} &\leq \zeta r e^{-\zeta \gammep /2}
			\leq  \frac{20}{\gammep}\left(  \log r  + \log \frac{1}{\gammep} \right) r e^{-10 \log r + 10\log \frac{1}{\gammep}} \\
			&\leq 40r^2 \gammep ^{-2}\cdot r^{-10}\gammep^{10} < 40\left(\gammep r^{-1} \right)^{8}    < 40\cdot 4^{-8} < \frac 12.
		\end{align*}
		
		Clearly, $S\notin\mathcal{U}$ only if $X\geq 1$. By Markov's inequality, the probability that $X \geq 1$ is at most $1/2$, so $\pr{S \notin \mathcal{U}} \leq \frac 12$, as claimed.
		
		To complete the proof of the claim, assume that
		\begin{equation} \label{eq:unibalanced-prob}
			\left| \mathcal{U} \cap \binom{[n]}{k} \right| \leq \frac 14 \binom nk \text{ \quad for all \quad } k \leq C.
		\end{equation}
		It follows that
		\begin{equation*}
			\pr{S \in \mathcal{U}} \leq \sum_{k=0}^{C} \pr{ S \in \mathcal U \;\middle|\; |S| = k} + \pr{|S| > C}
			\leq \sum_{k=0}^{C} \frac 14 \cdot \pr{|S| = k} + \frac 18 \leq \frac 38,
		\end{equation*}
		where the second inequality follows from~\eqref{eq:unibalanced-prob} and the Chernoff bound. We reached a contradiction, completing the proof of the Claim.
	\end{proof}
	Hence, for some $k$-vertex unibalanced graph $H$, $G$ contains at least $\frac 14 r^{-C^2} \binom nk$ copies of $H$; to see this, note that there are at most $r^{C^2}$ options for $H$.
	Applying Lemma~\ref{l:r-partite-coloured}, we obtain a homogeneous $\alpha \log n$-blow-up of $H$, as required.
\end{proof}

    \section{Concluding remarks}\label{sec:concludingremarks}
    
    \textbf{Asymptotics.} It remains an interesting open problem to improve the quantitative estimates from Theorem~\ref{thm:quarterbalanced} and Theorem~\ref{thm:anybalanced}. There is more room for improvement in Theorem~\ref{thm:quarterbalanced} compared to Theorem~\ref{thm:anybalanced}, but we believe both estimates should be quite far away from the truth. In particular, we don't see an inherent reason why the asymptotics of the locally balanced Ramsey problem should be different from the globally balanced Ramsey problem. That is, we believe that there should be an absolute constant $C$ so that when $n\geq (1/\eps)^{Ct}$, the conclusions of Theorems~\ref{thm:quarterbalanced} and ~\ref{thm:anybalanced} hold. Such a bound would be of the same order of magnitude as the bound from Theorem~\ref{thm:foxsudakov}. This bound would also be tight up to the value of $C$, as can be justified with a simple probabilistic construction. 
    \par The main obstacle to proving a bound of this form comes from our reliance on Lemma~\ref{l:r-partite-coloured}. Fox, Luo, and Wigderson~\cite{flw21} have recently combined the method of Nikiforov with ideas from graph regularity to obtain better estimates for a version of Lemma~\ref{l:r-partite-coloured} where the blow-up guaranteed is not necessarily homogeneous. Their ideas could quite possibly be modified to guarantee homogeneous blow-ups, but the resulting bound would still be rather far from being able to prove optimal bounds for Theorems~\ref{thm:quarterbalanced} and \ref{thm:anybalanced}. 
    
    \par \textbf{The extremal aspect.} A fruitful direction of research in the globally balanced version of the problem has been finding patterns in globally $\eps$-balanced $K_n$ where $\eps$ is a function of $n$. This seems to be an intriguing direction in the locally balanced setting as well, especially in the setting of Theorem~\ref{thm:anybalanced}. Namely, we raise the following problem. Let $t\geq 1$ be some integer. What is the smallest function $\eps:=\eps(n,t)$ as $n\to \infty$ such that any locally $\eps$-balanced $2$-coloured $K_n$ contains a $t$-blow-up of $P_1$, $\overline{P_1}$ or $P_3$? It already seems like an interesting challenge to prove $\eps\leq n^{C/t^2}$ for some absolute constant $C$.

\section*{Acknowledgements}
    We are grateful to Kyriakos Katsamaktsis for stimulating discussions at the early stages of this project, especially for the formulation of Lemma~\ref{lem:alternatingmatchings}, and to Shoham Letzter for helpful comments on the manuscript.

\bibliographystyle{plain}
\bibliography{bib}

\begin{thebibliography}{10}

\bibitem{BLM}
Matthew Bowen, Ander Lamaison, and Alp M{\"u}yesser.
\newblock Finding unavoidable colorful patterns in multicolored graphs.
\newblock {\em The Electronic Journal of Combinatorics}, 27, 2020.

\bibitem{caro2021unavoidable}
Yair Caro, Adriana Hansberg, and Amanda Montejano.
\newblock Unavoidable chromatic patterns in 2-colorings of the complete graph.
\newblock {\em Journal of Graph Theory}, 97(1):123--147, 2021.

\bibitem{cf13}
David Conlon and Jacob Fox.
\newblock Graph removal lemmas.
\newblock In {\em Surveys in combinatorics 2013}, volume 409 of {\em London
  Math. Soc. Lecture Note Ser.}, pages 1--49. Cambridge Univ. Press, Cambridge,
  2013.

\bibitem{CM}
Jonathan Cutler and Bal{\'a}zs Mont{\'a}gh.
\newblock Unavoidable subgraphs of colored graphs.
\newblock {\em Discrete Mathematics}, 308(19):4396--4413, 2008.

\bibitem{es35}
Paul Erd\H{o}s and George Szekeres.
\newblock A combinatorial problem in geometry.
\newblock {\em Compositio Math.}, 2:463--470, 1935.

\bibitem{flw21}
Jacob Fox, Sammy Luo, and Yuval Wigderson.
\newblock Extremal and {R}amsey results on graph blowups.
\newblock {\em J. Comb.}, 12(1):1--15, 2021.

\bibitem{fox2008induced}
Jacob Fox and Benny Sudakov.
\newblock Induced {R}amsey-type theorems.
\newblock {\em Advances in Mathematics}, 219(6):1771--1800, 2008.

\bibitem{FS}
Jacob Fox and Benny Sudakov.
\newblock Unavoidable patterns.
\newblock {\em Journal of Combinatorial Theory, Series A}, 115(8):1561--1569,
  2008.

\bibitem{fs11}
Jacob Fox and Benny Sudakov.
\newblock Dependent random choice.
\newblock {\em Random Structures Algorithms}, 38(1-2):68--99, 2011.

\bibitem{girao2022two}
Ant{\'o}nio Gir{\~a}o and Robert Hancock.
\newblock Two {R}amsey problems in blowups of graphs.
\newblock {\em arXiv preprint arXiv:2205.12826}, 2022.

\bibitem{girao2022turan}
Ant{\'o}nio Gir{\~a}o and Bhargav Narayanan.
\newblock Tur{\'a}n theorems for unavoidable patterns.
\newblock In {\em Mathematical Proceedings of the Cambridge Philosophical
  Society}, volume 172, pages 423--442. Cambridge University Press, 2022.

\bibitem{kst54}
Thomas K\H{o}vari, Vera S\'{o}s, and P\'al Tur\'{a}n.
\newblock On a problem of {K}. {Z}arankiewicz.
\newblock {\em Colloq. Math.}, 3:50--57, 1954.

\bibitem{muyesser2020turan}
Alp M{\"u}yesser and Michael Tait.
\newblock Tur{\'a}n-and {R}amsey-type results for unavoidable subgraphs.
\newblock {\em Journal of Graph Theory}, 2020.

\bibitem{nagy2019supersaturation}
Zolt\'{a}n~L\'{o}r\'{a}nt Nagy.
\newblock Supersaturation of {$C_4$}: from {Z}arankiewicz towards
  {E}rd{\H{o}}s-{S}imonovits-{S}idorenko.
\newblock {\em European J. Combin.}, 75:19--31, 2019.

\bibitem{nikiforov08}
Vladimir Nikiforov.
\newblock Graphs with many {$r$}-cliques have large complete {$r$}-partite
  subgraphs.
\newblock {\em Bull. Lond. Math. Soc.}, 40(1):23--25, 2008.

\end{thebibliography}

\end{document}